\documentclass[12pt,reqno,fleqn]{amsart}  
\usepackage{tikz}
\usepackage{amsmath,amstext,amsthm,amssymb,amsxtra}
\usepackage[top=1.5in, bottom=1.5in, left=1.25in, right=1.25in]	{geometry}
\usepackage[T1]{fontenc}
\usepackage{lmodern}

\usepackage{graphics}
\graphicspath{{./images/}}
\usepackage{mathrsfs}

\usepackage{mathtools}
\mathtoolsset{showonlyrefs,showmanualtags}

\usepackage{hyperref} 
\hypersetup{
	colorlinks=true,       
	linkcolor=blue,          
	citecolor=magenta,        
	filecolor=magenta,      
	urlcolor=cyan           
}

\usepackage[msc-links]{amsrefs}

\newtheorem{theorem}{Theorem}[section]

\newtheorem{lemma}{Lemma}[section]
\newtheorem{corollary}{Corollary}[section]
\newtheorem{OldTheorem}{Theorem}
\theoremstyle{definition}
\newtheorem{definition}{Definition}[section]
\theoremstyle{definition}

\theoremstyle{remark}
\newtheorem{remark}{Remark}[section]

\def\dmax{{\rm dmax}}
\def\dmin{{\rm dmin}}
\def\card{{\rm card}}

\def\ZF{\ensuremath{\mathcal F}}

\def\supp{{\rm supp\,}}

\def\ZU{\ensuremath{\mathfrak U}}
\def\ZV{\ensuremath{\mathfrak V}}

\def\MM^d{\ensuremath{\mathfrak M}}
\def\MM{\ensuremath{\mathcal M}}

\def\ZZ{\ensuremath{\mathbb Z}}
\def\ZI{\ensuremath{\mathbf 1}}

\def\ZD{\ensuremath{\mathcal D}}

\def\ZR{\ensuremath{\mathbb R}}
\def\zR{\ensuremath{\mathcal R}}

\def\pr{\ensuremath{\mathrm {pr}}}

\def\ZD{\ensuremath{\mathcal D}}

\def\ZG{{\mathcal G\,}}
\def\ZC{{\mathcal {C}}}

\numberwithin{equation}{section}
\def\md#1#2\emd{\ifx0#1
	\begin{equation*} #2 \end{equation*}\fi  
	\ifx1#1\begin{equation}#2\end{equation}\fi   
	\ifx2#1\begin{align*}#2\end{align*}\fi   
	\ifx3#1\begin{align}#2\end{align}\fi    
	\ifx4#1\begin{gather*}#2\end{gather*}\fi  
	\ifx5#1\begin{gather}#2\end{gather}\fi   
	\ifx6#1\begin{multline*}#2\end{multline*}\fi  
	\ifx7#1\begin{multline}#2\end{multline}\fi  
	\ifx8#1\begin{multline*}\begin{split}#2\end{split}\end{multline*}\fi
	\ifx9#1\begin{multline}\begin{split}#2\end{split}\end{multline}\fi
}
\newcommand {\e }[1]{\eqref{#1}}
\newcommand {\lem }[1]{Lemma \ref{#1}}
\newcommand {\rem }[1]{Remark \ref{#1}}
\newcommand {\cor }[1]{Corollary \ref{#1}}

\newcommand {\trm }[1]{Theorem \ref{#1}}

\title{Quantitative estimates for the absolute convergence of wavelet-type series}
\author{Grigori A. Karagulyan}
\address{Institute of Mathematics of NAS of RA, Marshal Baghramian ave., 24/5, Yerevan, 0019, Armenia}
\address{Faculty of Mathematics and Mechanics, Yerevan State
	University, Alex Manoogian, 1, 0025, Yerevan, Armenia} 
\email{g.karagulyan@ysu.am}
\author{Gor A. Melkumyan}   
\address{Faculty of Mathematics and Mechanics, Yerevan State
		University, Alex Manoogian, 1, 0025, Yerevan, Armenia} 
\email {gormelqumyan6@gmail.com}

\thanks{The work was supported by the Higher Education and Science Committee
	of RA, in the frames of the research project 21AG‐1A045}

\begin{document}
	
	\begin{abstract}
	We establish new quantitative estimates for general systems of functions with wavelet-type dyadic structure. These estimates are applied to obtain the optimal growth of various types of Weyl multipliers for certain wavelet-type systems. Some of our results are sufficiently general to allow the orthogonality assumption to be removed. In particular, as a consequence of these estimates we show  that the condition
	\begin{equation*}
		\sum_{n=1}^\infty\frac{1}{nw(n)}<\infty
	\end{equation*}
	is necessary and sufficient for an increasing sequence $w(n)$ to be an almost everywhere unconditional convergence Weyl multiplier for an arbitrary wavelet-type system. We also prove that $\log n$ is an almost everywhere convergence Weyl multiplier for any rearranged wavelet-type system, and that this bound is optimal. 
	\end{abstract}
	
	\maketitle
	\section{Introduction}
	\subsection{A historical overview}
	Let 
	\begin{equation}\label{r39}
		\Phi=\{\phi_n:\, n=1,2,\ldots\}\subset L^2(a,b)
	\end{equation}
	be an orthonormal system, where $(a,b)$ is an interval, possibly infinite. Recall that a sequence of positive numbers $w(n)\nearrow\infty$ is called an almost everywhere (a.e.) convergence Weyl multiplier (briefly Weyl multiplier) for $\Phi$ if every series 
	\begin{equation}\label{a1}
		\sum_{n=1}^\infty a_n\phi_n(x),
	\end{equation}
	with coefficients satisfying 
	\begin{equation}\label{a3}
		\sum_{n=1}^\infty a_n^2w(n)<\infty
	\end{equation}
	converges almost everywhere (see \cite{KaSa} or \cite{KaSt}). 
	The Menshov-Rademacher classical theorem (\cite{Men},  \cite{Rad}) states that the sequence $\log^2 n$ is a Weyl multiplier for any orthonormal system. The sharpness of $\log^2 n$ in this statement was established by Menshov in the same paper \cite{Men}: he constructed an orthonormal system for which any positive nondecreasing sequence $w(n)=o(\log^2n)$ fails to be Weyl multiplier. The following definitions are standard in the theory of orthogonal series. 
	\begin{definition}
		A sequence of positive numbers $w(n)\nearrow\infty$ is said to be an almost everywhere (a.e.) convergence Weyl multiplier for the rearrangements (an RC-multiplier) of an orthonormal system $\Phi$ if it is a Weyl multiplier for any subsystem $\{\phi_{n_k}\}$, where $\{n_k\}$ is a sequence of distinct natural numbers (not necessarily increasing).
	\end{definition}
	\begin{definition}
		A sequence of positive numbers $w(n)\nearrow\infty$ is called an a.e. unconditional convergence Weyl multiplier (a UC-multiplier) for an orthonormal system $\Phi$ if,  under condition \e{a3} series \e{a1} converges a.e. after any rearrangement of its terms.
	\end{definition}
	\begin{remark}
		These definitions are originally given for orthonormal systems, but we will also use them for non-orthogonal systems.
	\end{remark}
	Observe that according to the Menshov-Rademacher theorem we have that $\log^2 n$ is an RC-multiplier for any orthonormal system $\Phi$ and the counterexample of Menshov tells us that $\log^2 n$ is optimal in this statement. Papers \cite{Orl, Tan} provide a necessary and sufficient condition for a sequence to be a UC-multiplier for all orthonormal systems. In particular, these results imply that the sequence $\log^2n(\log\log n)^{1+\varepsilon}$ with $\varepsilon >0$ is a UC-multiplier for any orthonormal system, whereas $\log^2n\log\log n$ is not a UC-multiplier for some orthonormal systems.
	
	The study of RC and UC multipliers for classical orthonormal systems is a longstanding topic in the theory of orthogonal series. It is well known that the trigonometric, Walsh, Haar and Franklin systems are convergence system, i.e. the sequence $w(n)\equiv 1$ is a Weyl multiplier for those systems. However, $w(n)\equiv 1$ fails to be an RC-multiplier for these systems. A.~N.~Kolmogorov \cite{Kol} was the first to observe this phenomenon for the trigonometric system, although he did not publish a proof of this fact. A proof was later provided by Z.~Zahorski \cite{Zag}. Afterward, developing A.~Zahorski's argument, P.~L.~Ul\cprime yanov \cite{Uly6, Uly7} established analogous results for Haar and Walsh systems. Using techniques based on the Haar system, A.~M.~Olevskii \cite{Ole} proved that this phenomenon is, in fact, typical for arbitrary complete orthonormal systems.
	Later, P.~L.~Ul\cprime yanov \cite{Uly1,Uly3,Uly4} determined the optimal growth of RC and UC multipliers for the Haar system. 	
	\begin{OldTheorem}[Ul\cprime yanov, \cite{Uly4, Uly3}]\label{OT1}
		The sequence $\log n$ is an RC-multiplier for the Haar system. Conversely, any  nondecreasing sequence $w(n)$ satisfying $w(n)=o(\log n)$ is not an RC-multiplier for the Haar system.
	\end{OldTheorem}
	\begin{OldTheorem}[Ul\cprime yanov, \cite{Uly1}]\label{OT2}
		A nondecreasing sequence $w(n)$ is a UC-multiplier for the Haar system if and only if it satisfies
		\begin{equation}\label{a4}
			\sum_{n=1}^\infty\frac{1}{nw(n)}<\infty.
		\end{equation}
	\end{OldTheorem}
	
	G.~G.~Gevorkyan in \cite{Gev,Gev1,Gev2} proved the analogue of \trm{OT2} for the Franklin, Ciesielski and Str\"omberg wavelet-type systems. In \cite{BaGe} R.~H.~Barchudaryan and G.~G.~Gevorkyan investigated the same problem for generalized Haar and Franklin systems. For the trigonometric and Walsh systems, the problem of characterizing RC and UC multipliers remains open. These problems were first posed in P.~Ulyanov’s well-known survey paper \cite{Uly5} (1964) and have since been revisited in several of his subsequent works (see \cite{Uly4}, p.~1041; \cite{Uly8}, p.~80; \cite{Uly9}, p.~57).
	We first note that the Menshov-Rademacher theorem implies that $\log^2 n$ is an RC-multiplier for both the trigonometric and Walsh systems. On the other hand, no sequence $w(n)=o(\log^2n)$ is known to be an RC-multiplier for either of these systems. 
	Lower bounds for RC and UC multipliers of the Walsh system have been studied in \cite{Boch, Nak5, Nak3, Tan3}. The best result to date, obtained independently by S.~Bochkarev \cite{Boch2,Boch} and H.~Nakata \cite{Nak3}, shows that if an increasing sequence $w(n)$ does not satisfy Ul'yanov’s property \eqref{a4}, then it is not a UC-multiplier for the Walsh system. For the trigonometric system, lower bounds for RC and UC multipliers have been investigated in \cite{Mor, Tan2, Nak1, Nak2, Seroj, Kar2, Kar3}. An analogue of the Bochkarev–Nakata theorem for the trigonometric system was recently established by the author in \cite{Kar2}.
	
We recall also the following result of Pole\v{s}\v{c}uk \cite{Pol}, which clarifies the relationship between RC and UC multipliers of an orthonormal system.
\begin{OldTheorem}\label{Pol}
	Let $R(x)$, $x\in [0,\infty)$, be a nondecreasing positive function, satisfying
	\begin{equation}
		\int_1^\infty R(x)^{-1}dx<\infty.
	\end{equation}
	If a nondecreasing sequence $w(n)$ is an RC-multiplier for an orthonormal system $\{\phi_n(x)\}_{n\ge1}$, then the sequence $R(w(n))$ is a UC-multiplier for this system.
\end{OldTheorem}
It follows from this result that if the sequence $\log n$ is an RC-multiplier for an orthonormal system, then any increasing sequence $w(n)$ satisfying \e{a4} is a UC-multiplier for the same system. In particular, in this setting \trm{OT1} implies \trm{OT2}.

	\begin{definition}
		A sequence of positive numbers $w(n)\nearrow\infty$ is said to be a SRC-multiplier for an orthonormal system \e{r39} if it is an RC-multiplier for every orthonormal system of non-overlapping polynomials of the form
		\begin{equation}\label{y22}
			f_n=\sum_{k\in G_n}a_k\phi_k,\quad n=1,2,\ldots,
		\end{equation}
	where $G_n$ are pairwise disjoint sets of positive integers.
	\end{definition}
	Observe that if a sequence $w(n)$ is an SRC-multiplier for an orthonormal system \e{r39}, then it is also an RC-multiplier for $\Phi$, since every subsequence $\{\phi_{n_k}\}$ can be viewed as a polynomial system \e{y22} with $G_k=\{n_k\}$.  
	SRC-multiplier problem were first considered in the papers \cite{Kar1,Kar4,Kar5,KaKa}. In particular, A.~Kamont and the author studied this problem for wavelet-type systems. To state  the results of \cite{KaKa} (\trm{OT4}), we recall the definition of a wavelet-type system. Let $(A,B)$ be an interval, possible infinite. Denote
	\begin{equation}
		\ZZ^2(A,B)=\{(n,j)\in \ZZ^2:\, 2^{-n}\le B-A,\quad j\cdot2^{-n}\in (A,B]\}
	\end{equation}
	A system of functions (not necessarily orthogonal)
	\begin{equation}\label{y21}
		\Phi=\{\phi_{n,j}:\,(n,j)\in \ZZ^2(A,B)\}\subset L^2(\ZR)\cap L^\infty(\ZR)
	\end{equation}
	is called wavelet-type system if for every $(n,j)\in \ZZ^2(A,B)$ we have
	\begin{align}
		&\supp(\phi_{n,j})\subset [A,B],\\
		&\int_\zR\phi_{n,j}(t)dt=0,\label{u65}\\
		&|\phi_{n,j}(t)|\le c2^{n/2}\cdot \xi\big(2^nt-j)\big),\quad x\in [A,B],\label{u43}\\
		&|\phi_{n,j}(t)-\phi_{n,j}(t')|\le c2^{n/2}(2^n|t-t'|)^\alpha \cdot \xi\big(2^nt-j)\big)\text { if }t,t'\in [A,B], \label{u44}
	\end{align}
	where 
	 \begin{equation}\label{u51}
	 	\xi(x)=\frac{1}{(1+|x|)^{1+\beta}},\quad x\in \ZR,
	 \end{equation}
	 and $0<\alpha,\beta \le 1$. In some cases, we will additionally require wavelet-type systems to be orthonormal or merely $L^2$-normalized. The two principal choices for the interval $(A,B)$ in this definition are the real line $\ZR$ and the interval $[0,1]$. In the former case, $\ZZ^2(\ZR)=\ZZ^2$, while for $(A,B)=(0,1)$, 
	 \begin{equation}
	 	\ZZ^2(0,1)=\{(n,j)\in \ZZ^2:\,n\ge 0,\quad 1\le j\le 2^n \}
	 \end{equation}
	 \begin{remark}\label{rem1}
	 	Examples of wavelet-type systems include classical wavelets and orthonormal spline systems of arbitrary order, defined either on $[0,1]$ or on the real line $\ZR$; in particular, the Franklin, Ciesielski, and Strömberg systems considered in \cite{Gev,Gev1,Gev2}. 
	 \end{remark}
	\begin{OldTheorem}[\cite{KaKa}]\label{OT4}
		The sequence $\log n$ is SRC-multiplier for every orthonormal wavelet-type system.
	\end{OldTheorem}
Combining \trm{OT4} with the Pole\v{s}\v{c}uk \trm{Pol} implies.
	\begin{OldTheorem}[\cite{KaKa}]\label{OT5}
		If a nondecreasing sequence $w(n)$ satisfies \e{a4}, then it is a UC-multiplier for a sequence of non-overlapping orthonormal polynomials 
		\begin{equation}\label{y23}
			f_k=\sum_{(n,j)\in G_k}a_{n,j}\phi_{n,j},\quad k=1,2,\ldots,
		\end{equation}
		where $G_k\subset \ZZ^2$ are pairwise disjoint.
	\end{OldTheorem}

	\subsection{New results}
	In the present paper, we establish certain quantitative estimates for wavelet-type systems of both convergence and divergence type. These results are then applied to the characterization of RC and UC multipliers for such systems. Some of the results are sufficiently general to allow the orthogonality assumption to be dropped.
	
	We now state the main results of the paper. We call a constant admissible if it is either an absolute constant or depends only on the parameters from the definition of wavelet-type system. Then the notation $a\lesssim b$ will stand for the inequality $a<cb$, where $c$ is an admissible constant. Let $\Phi\in L^1(\ZR)$ be an arbitrary function. For integers $m, l\in \ZZ$ we denote 
	\begin{equation}\label{x20}
		\Phi_{m,l}(x)=2^{m/2}\Phi(2^mx-l).
	\end{equation}
\begin{theorem}\label{T3}
	Let $\Phi\in L^2(\ZR)$ and suppose that $\{\Phi_{m_k,l_k},\, k=1,2,\ldots,N\}$ is an arbitrary choice of $N$ different functions from \e{x20}. Then for every positive numbers $c_k>0,\, k=1,2,\ldots,N$, the following inequality holds:
	\begin{align}
		\bigg\|\sum_{k=1}^Nc_k\Phi_{m_k,l_k}\bigg\|_2\lesssim \sqrt{\log
			N}\cdot \|\Phi\|_1\left(\sum_{k=1}^Nc_k^2\right)^{1/2}.\label{x23}
	\end{align}
\end{theorem}
\begin{theorem}\label{T1}
	Let the system \e{y21} satisfy \e{u43} and suppose that an increasing positive sequence $w(n)$, $n\in \ZZ$, satisfies
	\begin{equation}\label{y62}
		\sum_{n\in \ZZ}w(n)^{-1}<\infty.
	\end{equation}
Then from the condition 
\begin{equation}\label{y65}
	\sum_{n,j\in \ZZ}a_{n,j}^2w(n)<\infty 
\end{equation}
it follows that
\begin{equation}
	\sum_{n,j\in \ZZ}|a_{n,j}||\phi_{n,j}(x)|<\infty \text{ a.e. on }\ZR.
\end{equation}
\end{theorem}
\begin{theorem}\label{T4}
	Let \e{y21} be a normalized wavelet-type system on $[0,1]$ (possibly non-orthogonal). If an increasing sequence $w(n)$ satisfies
	\begin{equation}\label{y83}
		\sum_{n\ge 1}w(n)^{-1}=\infty.
	\end{equation}
	 then there exist coefficients $a_{n,j}$  such that
	\begin{equation}
		\sum_{(n,j)\in \ZZ^2(0,1)}a_{n,j}^2w(n)<\infty,
	\end{equation}
	while the series 
		\begin{equation}
			\sum_{(n,j)\in \ZZ^2(0,1)}a_{n,j}\phi_{n,j}(x)
		\end{equation}
almost everywhere diverges on $[0,1]$ after some rearrangement of its terms.
\end{theorem}
Note that a wavelet-type system $\{\phi_{k,j}\}$ on $[0,1]$ has a natural ordering. Namely, for $n=2^k+j$, $i=1,2,\ldots,2^k$, $k=0,1,2,\ldots,$ the $n$th element in this ordering is the function
\begin{equation}\label{y63}
	\phi_n(x)=\phi_{k,j}(x),\quad n\ge 2.
\end{equation}
This ordering in particular is used in the Haar and Franklin systems, as well as in general spline and wavelet systems on $[0,1]$. These systems also include the constant function $1$ as their first element.
\begin{corollary}\label{C1}
A nondecreasing sequence $w(n)$ is a UC-multiplier for a $L^2$-normalized wavelet-type system $\{\phi_n(x)\}_{n\ge 2}$ on $(0,1)$ (possibly non-orthogonal)  if and only if it satisfies 
Ul\cprime yanov's condition \e{a4}.
\end{corollary}
\begin{corollary}\label{C2}
	Let \e{y21} be an orthonormal wavelet-type system, satisfying conditions \e{u65}-\e{u44}. Then the sequence $\log n$ is an RC-multiplier for this system. Moreover, any  increasing sequence $w(n)=o(\log n)$ is not RC-multiplier for this system. 
\end{corollary}
\begin{remark}
	Only parts of these corollaries concern series that diverge on sets of positive measure. In fact, we prove almost everywhere divergence of these series.
\end{remark}
\begin{remark}
	Note that in all the results except \cor{C2}, the orthogonality condition is not required. 
\end{remark}
\begin{remark}
		\trm{OT5} implies the convergence part of the \cor{C2} as well as its analogue results for several wavelet-type systems proved by G.~G.~Gevorkyan in \cite{Gev,Gev1,Gev2}. In fact, the proof of the convergence parts of \cor{C1} as well as its analogues of \cite{Gev,Gev1,Gev2} is simple use of the Cauchy-Schwartz inequality to a sum of absolute values of series (see for example \cite{KaSa}, ch. 3, Theorem 17), while for the counterexample parts of those theorems is a subtle extension of the method provided by Nikishin and Ul\cprime yanov \cite{NiUl}. However, the proof of \trm{OT5} uses quite different non-trivial argument, where the basic role plays a good-lambda inequality of \cite{CWW}.   
\end{remark}
\begin{remark}
	In Theorems \ref{OT4} and \ref{OT5}, the growth rates of the admissible RC and UC multipliers are optimal within the class of all orthonormal systems of the form \eqref{y23}, since they are already optimal in the particular case of the Haar system (see Theorems \ref{OT1} and \ref{OT2}). However, this optimality does not persist for every specific orthonormal system of the form \eqref{y23}, even when the system is derived from the Haar system. An example of such a system is the Rademacher system.
\end{remark}

\section{Proof of \trm{T3}}
We denote by 
\begin{equation}
	\Delta_j^m=\left[\frac{j-1}{2^m},\frac{j}{2^m}\right),\quad j,m\in \ZZ,
\end{equation}
the dyadic intervals on the real line $\ZR$ and set 
\begin{equation}\label{y78}
	\ZD_m=\{\Delta_j^m:\, j\in \ZZ\},\quad \ZD=\bigcup_{m\in \ZZ}\ZD_m.
\end{equation}
If $\ZU$ is a finite collection of dyadic intervals, then the set $E=\cup_{\Delta\in \ZU}\Delta$ can be uniquely represented as a union of pairwise disjoint maximal dyadic intervals. Note that such maximal intervals need not belong to the original collection $\ZU$. We denote by $\dmax(\ZU)$ the family of maximal intervals obtained in this way. Clearly, 
	\begin{equation*}
		\bigcup_{\Delta\in \ZU}\Delta=\bigcup_{\Delta\in \dmax(\ZU)}\Delta.
	\end{equation*}
	We also denote by $\dmin(\ZU)$ the collection of intervals in $\ZU$ (we call those the minimal intervals of $\ZU$) that do not contain any other interval from $\ZU$. Let $\ZU$ and $\ZV$ be two finite collections of dyadic intervals. Then we write $\ZU\Subset \ZV$, if every interval $I\in \ZU$ is contained in some interval $J\in \dmin(\ZV)$. 
	Given dyadic interval $I$, we denote by $\pr(I)$ its parent interval, i.e.,  the unique dyadic interval of twice the length that contains $I$.
	\begin{lemma}\label{L0}
		Let  $\ZU$ be a finite collection of dyadic intervals contained in another dyadic interval $I$. If
		\begin{equation}\label{y24}
			\sum_{\Delta\in \ZU}\ZI_{\Delta}(x)\ge l\text{ for all } x\in I,
		\end{equation}
		then the cardinality of $\ZU$ satisfies the bound $\card(\ZU)\ge 2^{l}-1$.
	\end{lemma}
	\begin{proof}
		Observe that under the condition \e{y24}, the cardinality $\card(\ZU)$ is minimized when $\ZU$ consists of all dyadic intervals $J\subset I$ satisfying 
		\begin{equation}
			|J|\ge 2^{-l+1}\cdot |I|.
		\end{equation}
		In this case, the number of such dyadic intervals is exactly $2^{l}-1$.
	\end{proof}
	\begin{lemma}\label{L1}
			Let $\ZU$ be a finite collection of dyadic intervals in $\ZR$ with $\card(\ZU)= N\le 2^n$. Then there is a decomposition
\begin{equation}\label{y26}
				\ZU=\ZU'\cup \ZU'',\quad \ZU'\cap \ZU''=\varnothing,\quad  \ZU''\Subset \ZU',
\end{equation}
such that
		\begin{align}
			&S(x)=\sum_{\Delta \in \ZU'}\ZI_{\Delta}(x)\le 2n,\quad x\in \ZR,\label{x2}\\
			&\bigcup_{\Delta\in \ZU''}\Delta\subset \{x:\,S(x)=2n\},\label{x5}\\
			&|\{x:\,S(x)=2n\}|\le 2^{1-n} |\{S(x)\neq 0\}|.\label{x3}
			\end{align}
	\end{lemma}
	\begin{proof}
	Ordering the intervals in $\ZU$ by decreasing length and applying a stopping-time argument on level $2n$, we obtain collections of intervals $\ZU'$ and $\ZU''$ satisfying
	\begin{equation}
		\ZU=\ZU'\cup \ZU'',\quad \ZU'\cap \ZU''=\varnothing,
	\end{equation}
	as well as conditions \e{x2} and \e{x5}. One can verify that the set $\{x:\,S(x)=2n\}$ is a union of certain minimal intervals of $\ZU'$. Moreover, by \e{x5} and the stopping-time construction every element in $\ZU''$ is contained in one of these minimal intervals. This implies that $\ZU''\Subset \ZU'$ and it remains to prove the bound \e{x3}. Let 
	\begin{equation}\label{x18}
		\{x:\, S(x)=2n\}=\bigcup_kJ_k
	\end{equation}
be a decomposition into maximal dyadic intervals. For each $J_k$ the sum
\begin{equation}
	\sum_{\Delta\in \ZU',\, \Delta\not\subset J_k}\ZI_\Delta(x)
\end{equation}
is constant on the parent interval $\pr(J_k)$. Let $l_k$ be the value of this constant. Then it is clear that 
	\begin{equation}\label{x17}
		|J_k|\le2^{-l_k} \cdot |\{x:\,S(x)\neq 0\}|\text{ and }l_k<2n.
	\end{equation}
	On the other hand for every $x\in J_k$ we have
		\begin{align*}
			\sum_{\Delta\in \ZU',\, \Delta\subset J_k}\ZI_\Delta(x)&=S(x)-\sum_{\Delta\in \ZU',\, \Delta\not\subset J_k}\ZI_\Delta(x)\\
			&=2n-l_k.
		\end{align*}
	Thus, applying \lem{L0}, we conclude that the number of intervals from $\ZU'$ contained in $J_k$ is not less than
\begin{equation*}
			2^{2n-l_k}-1\ge 2^{2n-l_k-1}. 
\end{equation*}
		Thus we obtain that
		\begin{equation*}
			\sum_{k}2^{2n-l_k-1}\le N\le 2^n.
		\end{equation*}
	Therefore, using also \e{x18} and \e{x17}, we get
		\begin{align*}
			|\{x:\, S(x)=2n\}|&= \sum_{k}|J_k|\\
			&\le |\{x:\,S(x)\neq 0\}|\sum_{k} 2^{-l_k} \\
			&\le 2^{1-n}\cdot |\{x:\,S(x)\neq 0\}|,
		\end{align*}
		completing the proof of lemma.
	\end{proof}
	
	\begin{lemma}\label{L6}
	Let $\ZU$ be a finite collection of dyadic intervals with $\card(\ZU)= N\le 2^n$. Then there exists a decomposition
	\begin{equation}\label{x10}
		\ZU=\bigcup_{k=1}^l\ZU_k,
	\end{equation}
	into pairwise disjoint collections $\ZU_k$, such that 
	\begin{align}
		&\ZU_{k}\Subset \ZU_{k-1},\quad k\ge 2,\label{x4}\\
		&S_k(x)=\sum_{\Delta\in \ZU_k}\ZI_{\Delta}(x)\le 2n, \quad k\ge 1, \label{x1}\\
		&\bigcup_{\Delta\in \ZU_{k}}\Delta\subset \{x:\,S_{k-1}(x)=2n\}, \quad k\ge 2,\label{x6}\\
		&|\{x:\,S_{k}(x)=2n\}\cap J|\le 2^{1-n}\cdot |J|,\quad k\ge 2.\label{x7}
	\end{align}
	where the latter holds for every minimal interval $J\in \dmin(\ZU_{k-1})$.
	\end{lemma}
	\begin{proof}
	Using induction, we construct pairwise disjoint collections 
\begin{equation}\label{y28}
		\ZU_1,\ZU_2,\ldots,\ZU_m, \ZG_{m}, 
\end{equation}
which satisfy conditions \e{x4},\e{x1},\e{x6} and \e{x7} for all $k\le m$, and such that
	\begin{align*}
		&\ZU=\ZU_1\cup\ZU_2\cup \ldots\cup \ZU_m\cup \ZG_{m},\\
		&\ZG_{m}\Subset \ZU_m, \\
		&\bigcup_{\Delta\in \ZG_{m}}\Delta\subset \{x:\, S_{m}(x)=2n\}.
	\end{align*}
	We begin by applying \lem{L1} to the initial collection $\ZU$. This yields two subcollections $\ZU',\ZU''\subset \ZU$, satisfying the conclusions of the lemma. We set $\ZU_1=\ZU'$ define the intermediate collection $\ZG_1=\ZU''$. This establishes the base step of the induction. Assume now that the collections \eqref{y28} have been constructed for some $m$. Applying \lem{L1} to the collection $\ZG_m$, we obtain a decomposition $\ZG_m=\ZU'\cup\ZU''$ satisfying the properties required in \lem{L1}. We then set $\ZU_{m+1}=\ZU'$ and $\ZG_{m+1}=\ZU''$ thereby completing the induction step. The procedure terminates when $\ZG_{l+1}=\varnothing $. At this stage, we obtain collections $\ZU_1,\ZU_2,\ldots,\ZU_l$ satisfying conditions \e{x4}-\e{x7} .
		\end{proof}
	\begin{lemma}\label{L3}
		Let measurable sets $E_1\supset E_2\supset \ldots \supset E_N$ from $\ZR$ and  functions $f_k\in L^2(\ZR)$, $k=1,2,\ldots, N$, satisfy
		\begin{align}
			&\supp(f_k)\subset E_{k},\quad k=1,2,\ldots,N,\label{x11}\\
			& \left\|f_k\right\|_{L^2(E_{m+1})}^2\le \frac{1}{2}\cdot \|f_k\|_{L^2(E_{m})}^2,\quad k\le m<N.\label{x12}
		\end{align}
		Then 
		\begin{equation}
			\left\|\sum_{k=1}^Nf_k\right\|_2^2\le 3\sum_{k=1}^n\left\|f_k\right\|_2^{2}.
		\end{equation}
	\end{lemma}
	\begin{proof}
		If $1\le k<m\le N$, then using \e{x12}, we obtain
		\begin{align}
			\left|\int_0^1f_kf_m\right|=&\left|\int_{E_m}f_k f_m\right|\\
			&\le \|f_k\|_{L^2(E_{m})}\| f_m\|_{L^2}\le 2^{k-m} \|f_k\|_2\|f_m\|_2\\
			&\le 2^{k-m-1}(\|f_k\|_2^2+\|f_m\|_2^2).
		\end{align}
		Therefore,
		\begin{align*}
			\left\|\sum_{k=1}^Nf_k\right\|_2^2&\le \sum_{k=1}^N\|f_k\|_2 ^2+\sum_{m\neq k}\left|\int_0^1f_kf_m\right|\\
			&\le \sum_{k=1}^N\|f_k\|_2 ^2+\sum_{m\neq k}2^{|k-m|-1}(\|f_k\|_2^2+\|f_m\|_2^2)\\
			&\le 3\sum_{k=1}^N\|f_k\|_2 ^2.
		\end{align*}
	\end{proof}
We are now ready to prove the main lemma used in the proof of \trm{T3}. This lemma corresponds to the special case of \trm{T3} for the Haar system. 
\begin{lemma}[main]\label{L5}
	Let $\ZU=\{\Delta_{m_1},\Delta_{m_2},\ldots,\Delta_{m_N}\}$ be a collection of dyadic intervals of cardinality $N\ge 2$. Then for any choice of positive numbers $c_j>0,\, j=1,2,\ldots,N$, the following inequality holds:
	\begin{align}
		\bigg\|\sum_{j=1}^Nc_j\ZI_{\Delta_{m_j}}\bigg\|_2\lesssim \sqrt{\log
			N}\left(\sum_{j=1}^Nc_j^2\cdot |\Delta_{m_j}|\right)^{1/2}.\label{x16}
	\end{align}
\end{lemma}
\begin{proof}
Without loss of generality we can suppose that $n=\lceil\log N\rceil\ge 2$.  Applying \lem{L6} we obtain families $\ZU_k$, satisfying the conclusions of lemma. Then we consider the functions
\begin{equation}\label{x13}
	f_k(x)=\sum_{\Delta\in \ZU_k}c_\Delta\ZI_{\Delta}(x),\quad k=1,2,\ldots,l,
\end{equation}
and the sets 
\begin{equation}\label{y27}
	E_k=\bigcup_{\Delta\in \ZU_k} \Delta,\quad k=1,2,\ldots,l. 
\end{equation}
Observe that for every $k$ the function $f_k$ is constant on each interval $J\in \dmin(\ZU_{k})$. That is $f_k(x)=a_J$ as $x\in J\in \dmin(\ZU_{k})$. We also have $\ZU_{k+2}\Subset \ZU_k$ and so
\begin{equation*}
	E_{k+2}\subset \bigcup_{J\in \dmin(\ZU_k)}J.
\end{equation*}
  Thus, using also properties \e{x6} and \e{x7}, we obtain 
\begin{align}
	\left\|f_k\right\|_{L^2(E_{k+2})}^2&=\sum_{J\in \dmin(\ZU_{k})}|a_J|^2\left|J\cap E_{k+2}\right|\\
	&\le\sum_{J\in \dmin(\ZU_k)}|a_J|^2\left|J\cap \{S_{k+1}(x)=2n\}\right|\qquad\qquad  (\text{see }\e{x6})\\
	&\le 2^{1-n}\sum_{J\in \dmin(\ZU_k)}|a_J|^2\left|J\right|\le 2^{1-n}\left\|f_k\right\|_{L^2(E_{k})}^2\qquad (\text{see }\e{x7})\\
	&\le \frac{1}{2}\cdot \left\|f_k\right\|_{L^2(E_{k})}^2. \label{x73}
\end{align}
We now split the sequences \eqref{x13} and \eqref{y27} into two subsequences consisting of even and odd indices. It follows from \eqref{x73} that each of these subsequences satisfies property \eqref{x12} in Lemma~\ref{L3}. Clearly, they also satisfy \eqref{x11}. Therefore, by applying Lemma~\ref{L3}, we obtain
\begin{align}
		\bigg\|\sum_{\Delta\in \ZU}c_\Delta\ZI_{\Delta}\bigg\|_2^2&=\bigg\|\sum_k\sum_{\Delta\in \ZU_k}c_\Delta\ZI_{\Delta}(x)\bigg\|_2=\left\|\sum_{k=1}^Nf_k\right\|_2^2\\
		&\le 2\left\|\sum_{k\text{ is even }}f_k\right\|_2^2+2\left\|\sum_{k\text{ is odd}}f_k\right\|_2^2\\
		&\le 6\sum_{k=1}^N\left\|f_k\right\|_2^{2}.\label{x14}
\end{align} 
From \e{x1}, it follows that for every $x\in \ZR$ the sum in \e{x13} has at least $2n$ non-zero terms. Thus, using Cauchy-Schwarz inequality inequality,  we then obtain
\begin{align}
	\|f_k\|_2^2&=\left\|\sum_{\Delta\in \ZU_k}c_\Delta\ZI_{\Delta}(x)\right\|_2^2\\
	&\le 2n\sum_{\Delta\in \ZU_k}|c_\Delta|^2|\Delta|\lesssim\log N\sum_{\Delta\in \ZU_k} |c_\Delta|^2|\Delta|.\label{x15}
\end{align}
Combining \e{x14} and \e{x15}, we obtain the desired bound \e{x16}.
\end{proof}

\begin{proof}[Proof of \trm{T3}]
Approximating the function $\Phi$ by step functions, one can reduce the theorem to the case
\begin{equation}\label{x24}
	\Phi(x)=\sum_{j\in \ZZ}\xi_j \ZI_{[(j-1)/2^n,j/2^n)}(x), \text{ where }\|\Phi\|_1=2^{-n}\cdot \sum_{j\in \ZZ}|\xi_j|=1.
\end{equation}
Moreover, applying dilation by $2^n$, we can also suppose that $n=0$ in \e{x24}. Thus $\sum_j|\xi_j|=1$ and we obtain
\begin{align}
	\left(\sum_{k=1}^Nc_k\Phi_{m_k,l_k}(x)\right)^2&=\left(\sum_{k=1}^Nc_k\cdot 2^{m_k/2}\sum_{j\in \ZZ}\xi_j \ZI_{[j-1,j)}(2^{m_k}x-l_k)\right)^2\\
	&=\left(\sum_{j\in \ZZ}\xi_j \sum_{k=1}^Nc_k\cdot 2^{m_k/2}\ZI_{[j-1,j)}(2^{m_k}x-l_k)\right)^2\\
	&=\left(\sum_{j\in \ZZ}\xi_j \sum_{k=1}^Nc_k\cdot 2^{m_k/2}\ZI_{\Delta_{j+l_k}^{m_k}}(x)\right)^2\\
	&\le \sum_{j\in \ZZ}|\xi_j|\left(\sum_{k=1}^Nc_k\cdot 2^{m_k/2}\ZI_{\Delta_{j+l_k}^{m_k}}(x)\right)^2.\label{x21}
\end{align}
On the other hand from \lem{L5} it follows for any fixed $j\in \ZZ$ that
\begin{equation}\label{x22}
\left\|\sum_{k=1}^Nc_k\cdot 2^{m_k/2}\ZI_{\Delta_{j+l_k}^{m_k}}\right\|_2^2\lesssim \log N \cdot \sum_{k=1}^Nc_k^2.
\end{equation}
From \e{x21} and \e{x22} we obtain \e{x23}.
\end{proof}

\section{Tree system of functions}
Recall the following definition from \cite{Kar4}.
\begin{definition}
	A finite or countable family of functions on $\ZR$ is called a tree system if it admits an ordering $\{f_n:\, n=1,2,\ldots\}$ (which we call a canonical ordering) such that for any $n$ there exist pairwise disjoint sets $U_n^+,U_n^-\subset \ZR$ with the following properties: for any positive integers $n$ and $k<n$ one of the following three conditions holds:
	\begin{equation}\label{y15}
		U_n\subset U_k^+,\quad U_n\subset U_k^-,\quad U_n\cap U_k=\varnothing\quad (\text{where }U_n=U_n^+\cup U_n^-).
	\end{equation}
Moreover,
\begin{align}
	\{x:\,f_n(x)> 0\}\subset U_n^+,\quad \{x:\,f_n(x)< 0\}\subset U_n^-.\label{y17}
\end{align}
\end{definition}
\begin{definition}
	A partition of $\ZR$ is a collection of intervals $\ZF=\{I_j=[a_{j-1},a_j):\, j\in \ZZ\}$, where 
	\begin{equation}
		a_k\to +\infty,\quad a_{-k}\to -\infty \text{ as }k\to +\infty.
	\end{equation}
	For two partitions $\ZF$ and $\ZC$ we write $\ZF\prec \ZC$ if every intervals of $\ZF$ is a union of some intervals of $\ZC$.
\end{definition}
\begin{definition}
	An interval $I=[a,b)$ is said to be sign-preserving for a function $f\in C(\ZR)$, if one of these conditions
	\begin{equation*}
		f(x)\ge 0,\quad f(x)\le 0
	\end{equation*}
	holds for every $x\in I$.  A partition $\ZF$ is called sign-preserving for a function $f\in C(\ZR)$, is each interval from $\ZF$ is sign-preserving for $f$. 
\end{definition}
 
\begin{lemma}\label{L8}
	Let $\ZF_n=\{J_{n,j}:\, j\in \ZZ\}$ and $\ZC_n=\{I_{n,j}:\, j\in \ZZ\}$, $n\in \ZZ$, be sequences of partitions such that $\ZF_n\prec \ZC_n\prec\ZF_{n+1}$, and suppose that a function system 
	\begin{equation}\label{y14}
		\{f_{n,j},\quad n,j\in \ZZ\}
	\end{equation}
	satisfies the conditions
	\begin{align}
		&\supp (f_{n,j})\subset J_{n,j},\label{y30}\\
		&\ZC_n \text{ is a sign-preserving partition for all functions }f_{n,j},\, j\in \ZZ. 
	\end{align}
	Then the system \e{y14} is a tree system.
\end{lemma}
\begin{proof}
Consider an ordering 
	\begin{equation}
		\ldots, g_{-n},g_{-n+1}, \ldots,g_{0},g_{1},\ldots,g_m,\ldots
	\end{equation}
of the system \e{y14} with a property that a function $f_{n,j}$ precedes $f_{m,i}$ whenever $n<m$. Suppose that $g_k=f_{n,j}$. Since $\ZC_n$ is a sign-preserving partition for $f_{n,j}$, we define $U_k^+$ to be the union of those intervals $I\in \ZC_n$ on which $f_{n,j}\ge 0$ and is not identically zero, and $U_k^-$ to be the union of those interval, where $f_{n,j}\le 0$ and is not identically zero. With this definition the conditions in \e{y17} are satisfied. Let $g_k=f_{s,j}$, $g_n=f_{t,i}$ and $k<n$. If $s=t$, then by \e{y30} we have $U_n\cap U_k=\varnothing$, thus the third condition in \e{y15} is satisfied. Otherwise we have $s<t$ and clearly one of two first conditions in \e{y15} holds. Indeed, by the definition both $U_k^+$ and $U_k^-$ are union of some elements of $\ZC_n$, while $U_n$ is contained in an interval in $\ZF_n$. Since $\ZF_n\prec \ZC_k$ we will either have $U_n\subset U_k^+$ or $U_n\subset U_k^-$. Therefore the system \e{y14} is a tree system.
\end{proof}
The following lemma is a version of a lemma from \cite{Kar4}.
\begin{lemma}\label{L14}
	If 
\begin{equation}\label{y20}
		\{f_k:\, k=1,2,\ldots,N\}
\end{equation}
	is a finite tree system on $\ZR$, then there exists a permutation $\sigma$ (a one to one mapping from the set $\{1,2,\ldots,N\}$ onto itself) such that 
	\begin{equation*}
		\max_{1\le p<q\le N}\left|\sum_{k=p}^qf_{\sigma(k)}(x)\right|\ge\frac{1}{2} \sum_{k=1}^N|f_k(x)|\text{ for all }x\in \ZR.
	\end{equation*}
\end{lemma}
\begin{proof}
	Without loss of generality, we may assume that \eqref{y20} is the canonical ordering of the given tree system. To construct the desired permutation $\sigma$ we rearrange the functions in \eqref{y20} so that
	\begin{align}
		&U_k\subset U^+_n \Rightarrow f_k\text{ precedes } f_n,\label{y31}\\
		&U_k\subset U^-_n \Rightarrow f_k\text{ follows } f_n.\label{y32}
	\end{align}
	Assume inductively we found an ordering for the first $m$ functions $f_j$ such that \eqref{y31} and \eqref{y32} hold for all indices $1\le k,n\le m$. We determine the position of $f_{m+1}$ relative to the functions 
	\begin{equation}\label{y75}
		f_1,f_2,\ldots,f_m
	\end{equation}
	as follows. We partition the functions in \eqref{y75} into three groups. The first groups consists of those $f_n$ for which $U_{m+1}\subset U_n^-$. The second group consists of those $f_k$ such that $U_{m+1}\subset U_k^+$. The remaining functions form the third group. Observe that for any $f_n$ from the first group and $f_k$ from the second group, we have, we have $U_k^+\cap U_n^-\neq\varnothing$. Then, by the definition of tree system (see \e{y15}), one of the following relations must hold:
	\begin{equation}\label{y74}
		U_k\subset U_n^-,\quad U_n\subset U_k^+.
	\end{equation}
	In view of \eqref{y31} and \eqref{y32}, both cases imply that $f_k$ follows $f_n$. Furthermore, if $f_j$ belongs to the third group, then by property \e{y15}, $U_{m+1}\cap U_j=\varnothing$. Consequently, we may place $f_{m+1}$ after all elements of the first group and before all elements of the second group, while its position relative to the third group is irrelevant. It is straightforward to verify that the relations \eqref{y31} and \eqref{y32} are then satisfied for all indices $1\le k,n\le m+1$. Hence we complete the induction that yields an ordering satisfying \eqref{y31} and \eqref{y32}. This defines the permutation $\sigma$ such that 
	\begin{align}
		&U_k\subset U^+_n \Rightarrow \sigma(k)<\sigma(n),\label{y76}\\
		&U_k\subset U^-_n \Rightarrow \sigma(k)>\sigma(n).\label{y77}
	\end{align}
	For any $x\in \ZR$ there are mutually disjoint subsets $A$ and $B$ of $\{1,2,\ldots,N\}$, such that 
	\begin{align*}
		&x\in U^-_{j},\text{ if }j\in A,\\
		&x\in U^+_{j},\text{ if }j\in B,\\
		&x\not\in U_j \text{ if }j\not\in A\cup B.
			\end{align*} 
	Observe that if $n\in A$ and $k\in B$ then $\sigma(k)>\sigma(n)$. Indeed, since $x\in U^+_k\cap U^-_n\neq \varnothing$, by property \e{y15} one of the conditions in \e{y74} holds.  Using \e{y76}, \e{y77} both cases of \e{y74} yield $\sigma(k)>\sigma(n)$. Since we also have $f_j(x)=0$ as $j\notin A\cup B$, for an appropriate integer $l=l(x)$ we have 
	 \begin{equation*}
	 	\sum_{k=1}^lf_{\sigma(k)}(x)=\sum_{j\in A} f_j(x),\quad \sum_{k=l+1}^nf_{\sigma(k)}=\sum_{j\in B} f_j(x).
	 \end{equation*}
	 Moreover, the terms in the first sum are all non-positive, while in the second sum are non-negative. This implies that 
	\begin{align}
\max_{1\le p<q\le n}\left|\sum_{k=p}^qf_{\sigma(k)}(x)\right|&\ge \max\left\{\sum_{j=1}^l |f_{\sigma(j)}(x)|, \sum_{j=l+1}^n |f_{\sigma(j)}(x)|\right\}\\
&\ge \frac{1}{2} \cdot \sum_{k=1}^m|f_k(x)|
	\end{align}
	that completes the proof of lemma.
\end{proof}

\section{Truncated wavelet-type system}
 
We use the notation $C[A,B]$ for the class of functions $f$ on $\ZR$ such that $\supp f\subset [A,B]$ and the restriction of $f$ to $[A,B]$ is continuous. 
Given wavelet-type system \e{y21} we consider truncated functions
\begin{align}
	&\bar\phi_{n,j}(x)=\phi_{n,j}(x)\cdot \ZI_{\{|\phi_{n,j}(t)|\ge \lambda\cdot 2^{n/2}\}}(x),\label{y52}\\
	&\bar{\bar\phi}_{n,j}(x)=\phi_{n,j}(x)\cdot \ZI_{\{|\phi_{n,j}(t)|< \lambda\cdot 2^{n/2}\}}(x)
\end{align}
\begin{lemma}\label{L4}
	If $\{\phi_{n,j}\}$ satisfies \e{u43} and \e{u44}, then
	
	1) $\bar \phi_{n,j}$ is sign-preserving on the intervals of the dyadic grid $\tau+\ZD_{n+\mu_0}$ for any $\tau\in \ZR$ provided that
	\begin{equation}\label{y18}
		2^{\mu_0}>(c/\lambda)^{1/\alpha}.
	\end{equation}
	
	2) we have
	\begin{equation}\label{y39}
		\supp (\bar \phi_{n,j})\subset \left[\frac{j-2^{\nu_0-1}}{2^n}+\tau,\frac{j+2^{\nu_0-1}}{2^n}+\tau\right),
	\end{equation}
	whenever
	\begin{equation}\label{y80}
		2^{\nu_0}>\frac{1}{4}\cdot (c/\lambda)^{\frac{1}{1+\beta}},\quad 0\le |\tau|\le 2^{\nu_0-n-2},
	\end{equation}
	where $c$, $\alpha$ and $\beta$ are the constants in \e{u43}-\e{u51}.
\end{lemma}
\begin{proof}
1) For simplicity, we consider the function $\phi_{0,0}(x)$. The general case follows by the change of variables $x\to 2^nt-j$.  Denote
	\begin{equation}
		\phi(x)=\phi_{0,0}(x),\quad \bar \phi(x)=\bar \phi_{0,0}(x)=\phi(x)\cdot \ZI_{\{|\phi(t)|\ge \lambda\}}(x).
	\end{equation}
	Hence we have 
		\begin{align}
		&|\phi(t)|\le c\cdot \xi(t),\quad t\in \ZR,\label{y4}\\
		&|\phi(t)-\phi(t')|\le c|t-t'|^\alpha \cdot \xi(t) \text{ for all }t,t'\in [A,B].\label{y3}
	\end{align}
	and let denote
	\begin{equation}
		\phi(x)=\phi_{0,0}(x),\quad \bar \phi(x)=\bar \phi_{0,0}(x).
	\end{equation}
	Without loss of generality we may assume that $\phi$ takes both positive and negative values, since otherwise the conclusion of the lemma holds trivially. Then, we consider maximal open intervals $J\subset(A,B)$, satisfying one of two following properties 
	\begin{align*}
		&\rm{I})\, \phi(x)>0 \ \text{for all }x\in J\, \&\,  \sup_{x\in J}\phi(x)>\lambda,\\
		&\rm{II})\,\phi(x)<0 \ \text{for all }x\in J\, \&\,  \inf_{x\in J}\phi(x)<-\lambda,
	\end{align*} 
	We denote by $\ZG^+$ and $\ZG^-$ the collections of such maximal intervals, satisfying either property I) or II), respectively. Set $\ZG=\ZG^+\cup \ZG^-$. We call an $J=(a,b)\in \ZG$ inner if $[a,b]\subset (A,B)$. There are at most two not inner intervals in $\ZG$. Moreover, if $J=(a,b)$ is not inner, then we will have either $a=A$ or $b=B$. We also note that $(a,b)=(A,B)$ is not possible, since otherwise $\ZG$ will consist of a single interval, contradicting the assumption that $\phi$ changes sign. Let $J=(a,b)\in \ZG$ and suppose that $J\in \ZG^+$. If $J$ is an inner interval, then by continuity, 
	\begin{align*}
		\phi(a)=\phi(b)=0.
	\end{align*}
	Define an interval $I=(a',b')\subset J$ by 
	\begin{equation}
		a'=\inf\{x\in (a,b):\, \phi(x)>\lambda\},\quad b'=\sup\{x\in (a,b):\, \phi(x)>\lambda\}.
	\end{equation}
	Again by continuity,
	\begin{equation}
		\phi(a')=\phi(b')=\lambda. 
	\end{equation}
	Applying \e{y3}, we obtain
	\begin{align}
		\lambda= \phi(a')-\phi(a)\le c(a'-a)^\alpha,\qquad \lambda= \phi(b)-\phi(b')\le c(b-b')^\alpha.
	\end{align}
	Hence
	\begin{align}
		a'-a\ge (\lambda/c)^{1/\alpha},\quad b-b'\ge (\lambda/c)^{1/\alpha}\label{y10}.
	\end{align}
	Observe that the same bounds hold also for the inner intervals $J\in \ZG^-$. Moreover, if $J$ is not an inner interval, then at least one of two conditions \e{y10} holds.  Therefore for all intervals $J\in \ZG$ we conclude that
	\begin{equation}\label{y29}
		|J|\ge a'-a+b-b'\ge (\lambda/c)^{1/\alpha}.
	\end{equation}
	Since $\phi$ satisfies \e{y3} and the intervals of $\ZG$ are pairwise disjoint, from \e{y29} it follows that there are only finitely many of those intervals. Hence we have 
	\begin{equation}
		\ZG=\{J_1,J_2,\ldots,J_m\},\quad J_k=(a_k,b_k), 
	\end{equation}
	and we assume that
	\begin{equation}
		-\infty\le a_1<b_1\le a_2<b_2\le\cdots\le a_m<b_m\le +\infty.
	\end{equation}
	As we have proved that for each interval $J_k$ there exists an interval $(a'_k,b'_k)\subset (a_k,b_k)$ such that at least one of two following conditions holds:
	\begin{equation}\label{y11}
		a'_k-a_k\ge (\lambda/c)^{1/\alpha},\quad b_k-b'_k\ge (\lambda/c)^{1/\alpha}.
	\end{equation}
	Introduce the intervals
	\begin{align}
		&U_0=(-\infty,a'_1),\, U_1=(b'_1,a'_2),\,\ldots, U_{m-1}=(b'_{m-1},a'_{m}),\, U_{m}=(b'_m,+\infty),\\
		&V_1=(a'_1,b'_1),\,V_2=(a'_2,b'_2),\ldots,V_m=(a'_m,b'_m).
	\end{align}
Altogether, these intervals form a finite partition of $\ZR$, disregarding the endpoints. From \e{y11} we have
	\begin{align}
		|U_k|\ge (\lambda/c)^{1/\alpha},\quad k=0,1,\ldots,m.
	\end{align}
	Moreover,
	\begin{align}
		\bar \phi(x)=0 \text{ as }x\in U_k,\quad k=0,1,\ldots,m,
	\end{align}
	and $\bar \phi$ is sign-preserving on each interval $V_k$, $k=1,2,\ldots,m$.
	Therefore any partition of $\ZR$ whose intervals have lengths strictly smaller than $(\lambda/c)^{1/\alpha}$ is sign-preserving for $\bar\phi$.
	In particular the dyadic grids $s+\ZD_{\mu_0}$ has this property whenever \e{y18} holds. This completes the proof of the first part of lemma. 
	
	2) To prove the second part we again consider the particular case $\phi(x)=\phi_{0,0}(x)$. Then we will need to show that
\begin{equation}\label{y88}
		\supp (\bar \phi)\subset \left[-2^{\nu_0-1}+\tau,2^{\nu_0-1}+\tau\right),\quad 0\le |\tau|\le 2^{\nu_0-2}.
	\end{equation}
It is enough to see that under the conditions \e{y80} (with $n=0$) the function $c\xi(x)$ in inequality \e{y4} 
has an upper bound $\lambda$ at the endpoints of the interval in \e{y88}.
\end{proof}
\begin{lemma}
		If $\{\phi_{n,j}\}$ is $L^2$-normalized wavelet-type system, then for any number $\varepsilon>0$, there exists a $\lambda>0$ such that 
	\begin{align}
		&\frac{\|\bar{\bar\phi}_{n,j}\|_1}{ \varepsilon }\le  c2^{-n/2}\le  \|\bar{\phi}_{n,j}\|_1\le C2^{-n/2},\label{y53} \\
		&|\{x:\, |\bar{\phi}_{n,j}(x)|> \lambda2^{n/2 }\}|>c2^{-n},\label{y81}
	\end{align}
	where $c$ and $C>0$ are admissible constants.
\end{lemma}
\begin{proof}
	Using \e{u43}, we obtain
	\begin{align}
		&1=\|\phi_{n,j}\|_2^2\lesssim 2^{n/2}\|\phi_{n,j}\|_1,\label{y47}\\
		&\|\phi_{n,j}\|_1\le 2^{n/2} \int_\ZR\xi(2^nx-j)dx=2^{-n/2}\|\xi\|_1\lesssim 2^{-n/2}.
	\end{align}
	This yields $\|\phi_{n,j}\|_1\sim 2^{-n/2}$. Thus, taking into account \e{u43} and \e{u51}, for small enough $\lambda>0$ we obtain 
	\begin{equation}
		2^{n/2}\|\bar{\bar\phi}_{n,j}\|_1<c\varepsilon,\quad 2^{n/2}\|\bar{\phi}_{n,j}\|_1>c,
	\end{equation}
	where $c>0$ is an admissible constant. This implies \e{y53}. Then, we have
	\begin{align}
		E=\{x:\, |\bar{\phi}_{n,j}(x)|> &\lambda2^{n/2 }\}\\
		&\subset U=\{x:\, 2^{n/2}\xi(2^nx-j)> \lambda2^{n/2 }/c\}.
	\end{align}
	On the other hand,
	\begin{equation}\label{y89}
		|U|=2^{-n}|\{x:\, \xi(x)> \lambda/c\}|\lesssim 2^{-n}\cdot \lambda^{-\frac{1}{1+\beta}}.
	\end{equation}
	Thus we obtain
	\begin{equation}
		 2^{n/2}\cdot|E|+\lambda 2^{n/2} (|U|-|E|)\ge \|\bar \phi_{n,j}\|_1\gtrsim 2^{-n/2}
	\end{equation}
	and therefore, using \e{y89}, 
	\begin{equation}
		|E|(1-\lambda)+2^{-n}\cdot \lambda^{\frac{\beta}{1+\beta}}\gtrsim 2^{-n}.
	\end{equation}
	For small enough $\lambda>0$ it follows that $|E|\gtrsim 2^{-n}$.
\end{proof}
	The integer constants
	\begin{equation}
		\nu_0=\left[(c/\lambda)^{1/\alpha}\right]+2,\quad \mu_0=\left[\frac{1}{4}\cdot (c/\lambda)^{\frac{1}{1+\beta}}\right]+2
	\end{equation}
satisfy \e{y18} and \e{y80}	respectively. Those, together with the constant 
\begin{equation}\label{y9}
	l=\mu_0+\nu_0-1>2,
\end{equation}
will be used throughout the remainder of the paper. At some moment those will be fixed when we fix the parameter $\lambda$. With these constants we associate the following subsystem of our wavelet-type system together with its truncations:
\begin{align}\label{y54}
	&\Psi_{k,j}(x)=\phi_{2^{kl},\,j2^{\nu_0} }(x),\\
	&\bar \Psi_{k,j}(x)=\bar \phi_{2^{kl},\,j2^{\nu_0} }(x),\quad  \bar{\bar \Psi}_{k,j}(x)=\bar {\bar\phi}_{2^{kl},\,j2^{\nu_0} }(x),\label{y64}
\end{align}
where 
\begin{equation}
	j\in G_k=\left\{j\in \ZZ:\,(2^{kl},\,j2^{\nu_0} )\in \ZZ^2(A,B)\right\},\quad k\in \ZZ.
	\end{equation}
By \lem{L4} we have
\begin{equation}\label{y56}
	\supp(\bar \Psi_{k,j})\subset U^\tau_{k,j}=\left[\frac{j-1/2}{2^{kl-\nu_0}}+\tau,\frac{j+1/2}{2^{kl-\nu_0}}+\tau\right), \, 0\le |\tau|\le 2^{\nu_0-kl-2}.
\end{equation}
Moreover, for any $\tau\in \ZR$ the dyadic grid $\tau+\ZD_{kl+\mu_0}$ is a sign-preserving partition for functions $\bar \Psi_{k,j}$ for every permitted integers $j$.
\begin{lemma}\label{L15}
If $\{\phi_{n,j}\}$ is a wavelet-type system on $[A,B]$, then the truncated function system $\{\bar \Psi_{k,j}(x)\}$ from \e{y54}
is a tree system on $\ZR$.
\end{lemma}
\begin{proof}
Set 
	\begin{align}
		\tau_k=\frac{1}{(2^l-1)2^{\mu_0+(k-1)l}},\quad k\in \ZZ,\label{y36}
	\end{align}
where $\mu_0$ and $l$ are constants in \e{y9}. One can check,
\begin{equation}\label{y38}
	0<\tau_k<2^{\nu_0-kl-2}.
\end{equation}
We remarked before the lemma that the dyadic grid  
\begin{equation}\label{y50}
	\ZC_k=\tau_k+\ZD_{kl+\mu_0}
\end{equation}
provides a sign-preserving partition for every function $\bar\Psi_{k,j}$, $j\in G_k$. On the other hand the collection of intervals $\ZF_k=\{U^{\tau_k}_{k, j }: j\in \ZZ\}$ in \e{y56} form a partition for $\ZR$. Moreover, one can check that $\ZF_k$ is a dyadic grid, namely
\begin{equation}
	\ZF_k=\{U^{\tau_k}_{k, j }: j\in \ZZ\}=\tau_k+\ZD_{kl-\nu_0+1}.
\end{equation}
Since $kl-\nu_0+1\le kl+\mu_0$, we conclude $\ZF_k\prec \ZC_k$. Then we have $\ZC_k= \ZF_{k+1}$ or equivalently
\begin{equation}\label{y37}
	\tau_k+\ZD_{kl+\mu_0}= \tau_{k+1}+\ZD_{(k+1)l-\nu_0+1},
\end{equation}
Indeed,  it follows from \e{y9} that 
\begin{equation*}
	r=kl+\mu_0=(k+1)l-\nu_0+1.
\end{equation*}
Thus \e{y37} can be equivalently written as 
\begin{equation}\label{y51}
	u+\ZD_r=\ZD_r,
\end{equation}
 where (see \e{y36})
\begin{align*}
	u=\tau_k-\tau_{k+1}=\frac{1}{2^{\mu_0}(2^l-1)}\left(\frac{1}{2^{(k-1)l}}-\frac{1}{2^{kl}}\right)=2^{-kl-\mu_0}=2^{-r}.
\end{align*}
This implies \e{y51} and therefore, \e{y37}. Hence we have 
\begin{equation}
	\ZF_k\prec \ZC_k=\ZF_{k+1}, \quad \supp(\Psi_{k,j})\subset U_{k,j}^{\tau_k}
\end{equation}
and $\ZC_k$ is a sign-preserving partition for every function $\bar\Psi_{k,j}$, $j\in G_k$. Thus applying \lem{L8}, we obtain that $\{\bar\Psi_{k,j}\}$ is a tree system. 
	\end{proof}
	Let $f\in L^1(\ZR)$ and $\Delta\subset \ZR$ be an arbitrary interval with $|\Delta|=2^{-m}$. Then for any integer $n\ge m$ we have 
	\begin{equation}\label{y61}
		\sum_{j\in \ZZ}\int_\Delta f(x+j\cdot 2^{-n})=2^{n-m}\int_\ZR f(t)dt=2^n|\Delta|\int_\ZR f(t)dt.
	\end{equation}
	This will be used in the following lemma and in some other instances below.
	\begin{lemma}\label{L12}
		Let $\{\phi_{n,j}\}$ be a $L^2$-normalized wavelet-type system on $[A,B]$ and let $\{\Psi_{k,j}(x)\}$ be as in \e{y54}. Denote
		\begin{align}
			&\bar S(x)= \sum_{k=m+1}^Ma_k\sum_{j\in G_k}|\bar \Psi_{k,j}(x)|,\\
			&\bar {\bar S}(x)= \sum_{k=m+1}^Ma_k\sum_{j\in G_k}|\bar{\bar \Psi}_{k,j}(x)|,
		\end{align} 
		where $a_k>0$ and $M>m$ are arbitrary integers such that $2^{-ml}<B-A$. Then we may fix a $\lambda$ as an admissible constant such that for any interval $\Delta\subset [A,B]$ with $|\Delta|\ge 2^{-ml}$, 
		\begin{align}
			|\{x\in \Delta:\, \bar S(x)>8\bar{\bar S}(x)\}|>c|\Delta|,
		\end{align}
		where $0<c<1$ is an admissible constant. 
	\end{lemma}
	\begin{proof}
	Without loss of generality we may assume that $|\Delta|=2^{-ml}$. From \e{u43}, \e{u51} and \e{y64} it follows that
		\begin{align}
			\sum_{j\in G_k}|\bar \Psi_{k,j}(x)|&\le \sum_{j\in G_k}|\phi_{2^{kl},\,j2^{\nu_0}}(x)|\\
			&\lesssim 2^{kl/2}\sum_{j\in \ZZ}\xi(2^{kl}x-j2^{\nu_0})\\
			&\lesssim 2^{kl/2}\sum_{j\in \ZZ}\xi(j2^{\nu_0})\\
			&\lesssim 2^{kl/2}.\label{y59}
		\end{align}
		Obviously, the supporting intervals 
		\begin{equation}
			U_{k,j}=U_{k,j}^0,\quad j\in \ZZ 
		\end{equation}
		from \e{y56} (with $\tau=0$) form a partition of $\ZR$ and $|U_{k,j}|=2^{\nu_0-kl}$. Thus, if $k> m$, then 
		\begin{equation}
			|\Delta|=2^{(k-m)l-\nu_0}|U_{k,j}|>2^{l-\nu_0}|U_{k,j}|>|U_{k,j}|.
		\end{equation}
		Then, using \e{y53}, we get
		\begin{align}
			\left\|\sum_{j\in G_k}|\bar \Psi_{k,j}(x)|\right\|_{L^1(\Delta)}&= 	\sum_{j\in G_k}\left\|\bar \phi_{2^{kl},\,j2^{\nu_0} }\right\|_{L^1(\Delta)}\\
			&\gtrsim 2^{-kl/2}\card\{j\in  G_k:\, U_{k,j}\cap \Delta\neq\varnothing\}\\
			&\gtrsim 2^{-kl/2}\cdot 2^{(k-m)l-\nu_0}\\
			&=2^{kl/2-\nu_0}\cdot|\Delta|.\label{y57}
		\end{align}
		Observe that for $\bar {\bar\xi}(x)=\min\{\xi(x),\lambda\}$ and for small enough $\lambda$ we have
		\begin{equation}
			\int_\ZR\bar {\bar\xi}(x)dx<\varepsilon.
		\end{equation}
		Thus, using the first inequality in \e{y53} and the remark concerning to \e{y61}, we obtain 
		\begin{align}
			\left\|\sum_{j\in G_k}|\bar {\bar \Psi}_{k,j}(x) |\right\|_{L^1(\Delta)}	&\lesssim 2^{kl/2}\sum_{j\in \ZZ}\int_\Delta\bar {\bar \xi}(2^{kl}x-j2^{\nu_0} )dx\\	
			&=2^{kl/2}\sum_{j\in \ZZ}\int_\Delta\bar {\bar \xi}(2^{kl}(x-j2^{\nu_0-kl}) )dx\\
			&=2^{kl/2}\cdot 2^{kl-\nu_0}|\Delta|\int_\ZR\bar {\bar \xi}(2^{kl}x)dx\\
			&=2^{kl/2}\cdot 2^{-\nu_0}|\Delta|\int_\ZR\bar {\bar \xi}(x)dx\\
			&< \varepsilon\cdot 2^{kl/2-\nu_0}\cdot|\Delta|.\label{y58}
		\end{align}
Thus, combining \e{y59}, \e{y57} and \e{y58}, we obtain positive admissible constants $c_1$, $c_2$ and $c_3$ such that $c_2<c_32^{\nu_0}$ and
\begin{equation}\label{y60}
	\frac{c_1}{\varepsilon}\cdot \|\bar{\bar S}(x)\|_{L^1(\Delta)}\le c_2\gamma|\Delta| \le \|\bar S(x)\|_{L^1(\Delta)},\quad   \|\bar S(x)\|_{L^\infty(\Delta)}\le c_3 \gamma\cdot 2^{\nu_0},
\end{equation}
where
\begin{equation}
	\gamma=\sum_{k=m+1}^Ma_k2^{kl/2-\nu_0}.
\end{equation}
Now consider the sets
		\begin{equation}
			E'_\Delta=\{x\in \Delta:\,\bar S(x)>c_2\gamma/2 \},\quad  E''_\Delta=\{x\in \Delta:\,\bar {\bar S}(x)<c_2\gamma/16 \}.
		\end{equation}
	To estimate the measure of the first set, observe that from an inequality in \e{y60} it follows the inequality
	\begin{equation}
		|E'_\Delta|\cdot c_3\gamma\cdot 2^{\nu_0} +(|\Delta|-|E'_\Delta|)\cdot c_2\gamma/2\ge \|\bar S(x)\|_{L^1(\Delta)}\ge  c_2\gamma|\Delta|.
	\end{equation}
Then solving this inequality yields 
\begin{equation}
		|E'_\Delta|>\frac{c_2}{2c_3\cdot 2^{\nu_0} -c_2}\cdot |\Delta|>\frac{c_2}{c_3\cdot 2^{\nu_0}}\cdot |\Delta|.
\end{equation}
For the second set, by Chebishev's inequality we have 
\begin{equation}
	|E''_\Delta|>|\Delta|-\frac{\|\bar {\bar S}\|_{L_1(\Delta)}}{c_2\gamma/16}\ge |\Delta|(1-16\varepsilon/c_1).
\end{equation}
Since $\varepsilon $ can be arbitrarily small, thus we obtain
\begin{align*}
		|\{x\in \Delta:\, \bar S(x)>8\bar{\bar S}(x)\}|&\ge |E'_\Delta\cap E''_\Delta|\ge  |E'_\Delta|+|E''_\Delta|-|\Delta|\\
		&\ge |\Delta|\left(\frac{c_2}{c_3\cdot 2^{\nu_0}}-16\varepsilon/c_1\right)>\frac{c_2}{2c_3\cdot 2^{\nu_0}}\cdot |\Delta|.
\end{align*}
This completes the proof of lemma.
	\end{proof}
	
	\section{Proofs of Theorems \ref{T1} and \ref{T4}}
	
	\begin{lemma}\label{L9}
		Let $\xi$ be the function in \e{u51} and denote 
	\begin{equation}
			\xi_{m,j}(x)=2^{m/2}\xi(2^mx-j),\quad n,j\in \ZZ.
	\end{equation}
	If an increasing sequence $w(n)$ satisfies condition \e{y62}, then for the coefficients satisfying \e{y65} it follows that
		\begin{equation}\label{y82}
			\sum_{n,j\in \ZZ}|a_{n,j}|\xi_{n,j}(x)<\infty\text{ a.e. on }\ZR.
		\end{equation}
	\end{lemma}
	\begin{proof}
		First, observe that for any number $a>0$,
			\begin{equation}\label{y44}
				\sum_{j\in \ZZ}\left(\int_{-a}^{a}\xi_{n,j}(t)dt\right)^2\lesssim a,\quad n\in \ZZ.
			\end{equation}
		Indeed, without loss of generality we can assume that $a= 2^{m}$, where $m>-n$ is an integer. Using \e{y61}, we obtain
			\begin{align*}
				\sum_{j\in \ZZ}\left(\int_{-2^{m}}^{2^{m}}\xi_{n,j}(t)dt\right)^2&=2^n \sum_{j\in \ZZ}\left(\int_{-2^{m}}^{2^{m}}\xi(2^nt-j)dt\right)^2\\
				&=2^{-n}\sum_{j\in \ZZ}\left(\int_{-2^{n+m}}^{2^{n+m}}\xi(t-j)dt\right)^2\\
				&\le2^{-n}\|\xi\|_{L^1(\ZR)}\sum_{j\in \ZZ}\int_{-2^{n+m}}^{2^{n+m}}\xi(t-j)dt\lesssim 2^{m}=a.
			\end{align*}
Thus \e{y44} follows. Now let $w(n)$ satisfy condition \e{y62} and we have \e{y65}. Applying \e{y44}, we get
		\begin{align}
			\sum_{n,j\in \ZZ}\int_{-a}^a&|a_{n,j}\xi_{n,j}(x)|dx\\
			&\le \sum_{n\in \ZZ}\left(\sum_{j\in \ZZ} |a_{n,j}|^2\right)^{1/2}\left(\sum_{j\in \ZZ}\left(\int_{-a}^a\xi_{n,j}(x)dx\right)^2\right)^{1/2}\\
			&\lesssim \sqrt{a}\sum_{n=0}^\infty\left(\sum_{j=1}^{2^k} |a_{n,j}|^2\right)^{1/2}\\
			&\le \sqrt{a}\left(\sum_{n,j\in \ZZ}a_{n,j}^2w(n)\right)^{1/2}\cdot \left(\sum_{n=0}^\infty \frac{1}{w(n)}\right)^{1/2}<\infty.
		\end{align}
		Since the latter holds for any $a>0$, we obtain \e{y82} almost everywhere.
	\end{proof}
	\begin{proof}[Proof of \trm{T1}] The proof immediately follows from \lem{L9}.
	\end{proof}
	\begin{lemma}\label{L13}
		Let $\{\ZF_k\}_{k\ge 1}$ be a sequence of finite partitions of $[0,1)$ such that 
		\begin{equation*}
			\lim_{k\to\infty}\max_{J\in \ZF_k}|J|=0
		\end{equation*}
		and suppose that $E_k\subset [0,1)$ is a sequence of measurable sets such that 
		\begin{equation}\label{y70}
			|E_k\cap F|>c|F|
		\end{equation}
		for every set $F\in \ZF_k$. If positive numbers $a_k>0$ satisfy $\sum_{k= 1}^\infty a_k=\infty$, then
		\begin{equation}\label{y69}
			\sum_{k=1}^\infty a_k\ZI_{E_k}(x)=\infty\text{ a.e..}
		\end{equation}
	\end{lemma}
	\begin{proof}
		Denote by $E$ the set in $[0,1)$, where the divergence in \e{y69} holds. Chose $m$ and an arbitrary set $G\in \ZF_m$. Then for $n>m$ denote 
		\begin{equation}
			S_{m,n}=\sum_{k=m}^n a_k\ZI_{E_k}(x)
		\end{equation}
		We have
		\begin{equation}
			\|S_{m,n}\|_\infty\le A_{m,n}=\sum_{k=m}^n a_k,	\quad \|S_{m,n}\|_{L^1(G)}\ge cA_{m,n}|G|,
		\end{equation}
		where the second bound follows from \e{y70}. Thus, a standard argument implies 
		\begin{equation}\label{y71}
			|\{x\in G:\, S_{m,n}(x)>c_1A_{m,n}\}|>c_2|G|,
		\end{equation}
		where $c_1,c_2$ are positive constants depending on $c$. Since $\lim_{n\to\infty}A_{m,n}=\infty$, condition \e{y71} yields 
		\begin{equation*}
			\lim_{n\to\infty}S_{m,n}(x)=\infty
		\end{equation*}
		on a subset of $G$ with a measure greater that $c_2|G|$. Therefore we obtain that $|E\cap G|>c_2|G|$ for every elements $G$ belonging to a partition $\ZF_n$. Thus the Lebesgue density property of $E$ implies $|E|=1$ that completes the proof of the lemma.
	\end{proof}
	\begin{remark}\label{R1}
		Observe that, under the assumptions of  \lem{L13} we also have $|\liminf_{n\to\infty }E_n|=1$, which follows by taking $a_k=1$ in the lemma.
	\end{remark}
	\begin{proof}[Proof of \trm{T4}]
	We consider an increasing sequence $w(n)$ satisfying \e{y83}. Denote $\bar w(k)=w(kl-\nu_0)$. Clearly, \e{y83} yields
	\begin{equation}
	\sum_{k=1}^\infty \bar w(k)^{-1}=\infty.
	\end{equation}
	One can find a sequence $q_k\nearrow \infty$ such that
	\begin{equation}\label{y66}
		\sum_{k=1}^\infty\frac{1}{\bar w(k)q_k}=\infty,\quad \sum_{k=1}^\infty\frac{1}{\bar w(k)q_k^2}<\infty.
	\end{equation}
	Then we consider series 
	\begin{equation}\label{y67}
		\sum_{k=1}^\infty\sum_{j=1}^{2^{kl-\nu_0}}\frac{1}{2^{kl/2}\bar w(k)q_k}\Psi_{k,j}(x)
	\end{equation}
where $\Psi_{k,j}$ is the subsystem \e{y64} of our wavelet-type system. Note that the coefficients of this series satisfy \e {y65}, since by the second relation \e{y66} we have
\begin{equation}
	\sum_{k=1}^\infty\sum_{j=1}^{2^{kl-\nu_0}}\left(\frac{1}{2^{kl/2}\bar w(k)q_k}\right)^2\bar w(k)=2^{-\nu_0}\sum_{k=1}^\infty\frac{1}{\bar w(k)q_k^2}<\infty.
\end{equation}
Thus it remains to prove that series \e{y67} is almost everywhere divergent after some rearrangement of its terms. Let $\bar \Psi_{k,j}$ be the upper $\lambda$-truncation of $\Psi_{k,j}$ and suppose that $U_{k,j}$ is the corresponding supporting interval in \e{y56}. Denote
	\begin{align}
		&\bar U_{k,j}=[0,1)\cap U_{k,j},\\
		&E_k=\left\{x\in [0,1]:\,\sum_{j=1}^{2^{kl-\nu_0}}|\bar \Psi_{k,j}(x)|>\lambda2^{kl/2}\right\}.
	\end{align}
	Observe that, by definition, some intervals $U_{k,j}$ may not be contained in $[0,1)$. In fact, this occurs only for the edge intervals $U_{k,1}$ and $U_{k,2^{kl-\nu_0}}$. Nevertheless,  it follows from \e{y81} that 
	\begin{equation}
		|E_k\cap \bar U_{k,j}|=|\{\Psi_{k,j}(x)|>\lambda2^{kl/2}\}|\ge c|\bar U_{k,j}|\text{ for all }j=1,2,\ldots,2^{kl-\nu_0}
	\end{equation}
	Thus the sequence of partitions $\ZF_k=\{\bar U_{k,j}:\, j\in \ZZ\}$ of $[0,1)$ satisfies the hypothesis of \lem{L13}. Then, taking into account \e{y66}, we conclude that 
	\begin{equation}
		\sum_{k=1}^\infty\frac{\ZI_{E_k}(x)}{\bar w(k)q_k}=\infty\text{ a.e. on }[0,1).
	\end{equation}
	By the definition of the sets $E_k$, the latter clearly implies
	\begin{equation}
		\sum_{k=1}^\infty\sum_{j=1}^{2^{kl-\nu_0}}\frac{1}{2^{kl/2}\bar w(k)q_k}|\bar \Psi_{k,j}(x)|=\infty \text{ a.e. on }[0,1).
	\end{equation}
	Thus one can find an increasing sequence of integers $m_k$ such that the sets
	\begin{equation}
		F'_k=\left\{\sum_{j=m_k+1}^{m_{k+1}}\sum_{j=1}^{2^{kl-\nu_0}}\frac{1}{2^{kl/2}\bar w(k)q_k}|\bar\Psi_{k,j}(x)|>k\right\}
	\end{equation}
	satisfy
	\begin{equation}\label{y84}
		|F'_k\cap \Delta|>(1-\varepsilon)|\Delta|\text{ for every set }\Delta\in \ZF_{m_k},
	\end{equation}
	where $\varepsilon>0$ can be arbitrary small. Consider the collection of functions
	\begin{equation}
		\frac{1}{2^{kl/2}\bar w(k)q_k}\bar \Psi_{k,j}(x)\text{, where }m_s\le k<m_{s+1},\, j=1,2,\ldots 2^{kl-\nu_0},
	\end{equation}
	which number is equal to $n_s=(m_{s+1}-m_s)2^{kl-\nu_0}$. According to \lem{L15} this collection form a tree system. Then applying \lem{L14}, we find a numeration of this collection,
	\begin{equation}
		\bar \psi_{s,j}(x),\quad j=1,2,\ldots,n_s,
	\end{equation} 
	such that
	\begin{align}
		\max_{1\le m\le n_s}\left|\sum_{j=1}^{m}	\bar \psi_{s,j}(x)\right|&\ge \frac{1}{2}\max_{1\le p<q\le n_s}\left|\sum_{j=p}^{q}	\bar \psi_{s,j}(x)\right|\\
		&\ge \frac{1}{4}\sum_{j=1}^{n_s}	|\bar \psi_{s,j}(x)|, \quad x\in [0,1].
	\end{align}
	Set
	\begin{equation}
		F''_k=\left\{x:\,\sum_{j=1}^{n_s}	|\bar \psi_{s,j}(x)|> 8\sum_{j=1}^{n_s}	|\bar {\bar \psi}_{s,j}(x)|\right\}
	\end{equation}
	By  \lem{L12} we have
	\begin{equation}\label{y85}
		|F''_k\cap \Delta|>c|\Delta|\text{ for every set }\Delta\in \ZF_{m_k},
	\end{equation}
	since the intervals of $\ZF_{m_k}$ have length greater that $2^{-m_kl}$. Thus for every point $x\in F_k=F'_k\cap F''_k$ we have 
	\begin{align}
		\max_{1\le m\le n_s}\left|\sum_{j=1}^{m}	\psi_{s,j}(x)\right|&\ge \max_{1\le m\le n_s}\left|\sum_{j=1}^{m}	\bar \psi_{s,j}(x)\right|- \sum_{j=1}^{n_s}	|\bar {\bar \psi}_{s,j}(x)|\\
		&\ge  \frac{1}{4}\sum_{j=1}^{n_s}	|\bar \psi_{s,j}(x)|- \sum_{j=1}^{n_s}	|\bar {\bar \psi}_{s,j}(x)|\\
		&\ge  \frac{1}{8}\sum_{j=1}^{n_s}	|\bar \psi_{s,j}(x)|>\frac{s}{8}\text{ if }x\in F_k=F'_k\cap F''_k.
	\end{align}
	On the other hand, from \e{y84} and \e{y85}, we have 
\begin{equation}\label{y86}
	|F_k\cap \Delta|>c|\Delta|\text{ for every }\Delta\in \ZF_{m_k}
\end{equation}
with a constant $c>0$ and we have
	\begin{equation}
		\max_{1\le m\le n_s}\left|\sum_{j=1}^{m}	\psi_{s,j}(x)\right|\ge\frac{s}{8} \text{ as } x\in F_k.
	\end{equation}
Thus we obtain
	\begin{equation}
		\lim_{s\to \infty}\max_{1\le m\le n_s}\left|\sum_{j=1}^{m}	\psi_{s,j}(x)\right|=\infty \text{ as }x\in E=\liminf_{k\to\infty } F_k.
	\end{equation}
	On the other hand by \lem{L13} (see also \rem{R1}) it follows that $|E|=1$. Thus the series 
	\begin{equation}
		\sum_{s=1}^\infty \sum_{j=1}^{n_s}\psi_{s,j}(x)
	\end{equation}
	diverges almost everywhere. Since this series is a rearrangement of the terms of \e{y67}, the proof of the theorem is complete.
	\end{proof}
	
	\section{Proofs of Corollaries \ref{C1} and \ref{C2}}
\begin{proof}[Proofs of \cor{C1}]. The proofs of the convergence and divergence parts of \cor{C1} immediately follows from Theorems \ref{T1} and \ref{T4} respectively.
\end{proof}
\begin{proof}[Proofs of \cor{C2}] 
	\textit{Convergence part: }Having \trm{T3}, the proof of the first part of  \cor{C2} is based on a simple standard argument (see for example \cite{KaSt}). Hence, let $\{\phi_n=\phi_{m_n,l_n}\}$ be an arbitrary sequence, consisting of distinct functions from our wavelet-type system and suppose that 
	\begin{equation}\label{y42}
		\sum_{n=1}^\infty a_n^2\log n<\infty.
	\end{equation}
	We need to prove almost everywhere convergence of the series 
	\begin{equation}\label{y40}
		\sum_{n=1}^\infty a_n\phi_n(x).
	\end{equation}
	Denote
	\begin{equation}
		\delta_k(x)=\max_{2^k< m\le2^{k+1}}\left|\sum_{j=2^k}^ma_j\phi_j(x)\right|,\quad k=0,1,\ldots.
	\end{equation}
	Applying \trm{T3} with $\Phi(x)=\xi(x)$, we get
	\begin{align}
		\left\|\delta_k\right\|_2&\le \left\|\sum_{j=2^k+1}^{2^{k+1}}|a_j\phi_j(x)| \right\|_2\lesssim \left\|\sum_{j=2^k+1}^{2^{k+1}}|a_j|\Phi_{m_n,l_n}\right\|_2^2\lesssim k\sum_{j=2^k+1}^{2^{k+1}}a_j^2.
	\end{align}
Therefore,
\begin{align*}
	\sum_{k=1}^\infty \int_\ZR\delta_k^2(t)dt\lesssim \sum_{k=1}^\infty k\sum_{j=2^k+1}^{2^{k+1}}a_j^2\lesssim\sum_{n=1}^\infty a_n^2\log n<\infty.
\end{align*}
This implies that
\begin{equation}\label{y41}
	\lim_{k\to\infty}\max_{2^k\le m<2^{k+1}}\left|\sum_{j=2^k}^ma_j\phi_j(x)\right|=0\text{ almost everywhere.}
\end{equation}
Let $S_n(x)$ be the $n$th partial sum of series \e{y40} and suppose that it converges in $L^2$-norm to a function $f\in L^2(\ZR)$. Then, using \e{y42}, we obtain
\begin{equation}
	\sum_{k=1}^\infty\int_\ZR\left|f(t)-S_{2^k}(t)\right|^2dt=\sum_{k=1}^\infty \sum_{j=2^{k+1}+1}^\infty a_j^2\lesssim \sum_{j=1}^\infty a_j^2\log j<\infty.
\end{equation}
This implies almost everywhere convergence of the sequence $S_{2^k}(x)$. Combining this and \e{y41}, we obtain almost everywhere convergence of series \e{y40}.

\textit{Divergence part:} 
The argument demonstrated in the proof of the second part of \trm{T1} can be similarly applied in this part of the theorem. Nevertheless,  we will proceed, using the result of \trm{Pol}. Suppose to the contrary that a nondecreasing sequence $w(n)$ is an RC-multiplier for a system satisfying the requirements of the theorem, while we have $w(n)=o(\log n)$ as $n\to \infty$. Then we can write
\begin{equation}
	w(2^k)=\varepsilon(k)k,\text{ where }\lim_{k\to\infty}\varepsilon(k)=0.
\end{equation}
Thus, one can find a nondecreasing function $R(x)$, $x\in [0,\infty)$ such that
\begin{align}
	&\sum_{k=1}^\infty R(k)^{-1}<\infty,\\
	&\sum_{k=1}^\infty R(w(2^k))^{-1}= \sum_{k=1}^\infty R(\varepsilon(k)\cdot  k)^{-1}=\infty.\label{y87}
\end{align}
Since $w(n)$ is an RC-multiplier for our system, using \trm{Pol} we can say the $u(n)=R(w(n))$ is a UC-multiplier for it. On the other hand from \e{y87} it follows that
\begin{equation}
	\sum_{k=1}^\infty u(2^k)^{-1}=\sum_{k=1}^\infty R(w(2^k))^{-1}=\infty.
\end{equation}
Thus using \trm{T1} we obtain that $u(n)$ is not a UC-multiplier for $\Phi$, which is a contradiction.
\end{proof}

	\bibliographystyle{plain}
	

\end{document}